\documentclass[12pt,twoside, a4paper, english, reqno]{amsart}
\usepackage[dvips]{epsfig}
\usepackage{amscd}
\usepackage{a4wide}
\usepackage{amssymb}
\usepackage{amsthm}
\usepackage{amsmath}
\usepackage{bm} 
\usepackage{latexsym}

\usepackage{upref}
\usepackage{hyperref}
\usepackage{ulem}
\usepackage{enumerate}
\usepackage[usenames,dvipsnames]{xcolor}
\usepackage{subcaption}

\theoremstyle{plain}

\newtheorem{theorem}{Theorem}[section]
\newtheorem{proposition}[theorem]{Proposition}

\numberwithin{figure}{section}
\numberwithin{table}{section}

\allowdisplaybreaks

\newcounter{asnr}

{\ifnum\value{asnr}=0 \stepcounter{asnr} 
	\begin{enumerate}[label=\textbf{A}.\arabic{enumi}]
		\else
		\begin{enumerate}[label=\textbf{A}.\arabic{enumi},resume] \fi}
		{\end{enumerate}}
	
	\newcounter{defnr}

	{\ifnum\value{defnr}=0 \stepcounter{defnr} 
		\begin{enumerate}[label=\textbf{D}.\arabic{enumi}]
			\else
			\begin{enumerate}[label=\textbf{D}.\arabic{enumi},resume] \fi}
			{\end{enumerate}}
		
		\newcommand{\R}{\mathbb{R}}
		\newcommand{\N}{\mathbb{N}}
		\newcommand{\pl}{\partial}
		\newcommand{\eps}{\varepsilon}
		
\newcommand\smallO{
	\mathchoice
	{{\scriptstyle\mathcal{O}}}% \displaystyle
	{{\scriptstyle\mathcal{O}}}% \textstyle
	{{\scriptscriptstyle\mathcal{O}}}% \scriptstyle
	{\scalebox{.7}{$\scriptscriptstyle\mathcal{O}$}}%\scriptscriptstyle
}
		
		\numberwithin{equation}{section} \allowdisplaybreaks
		
\title[Boundary determination]{
Boundary determination of coefficients appearing in a perturbed weighted $p$-Laplace equation}
		
\author[N. Kumar, T. Sarkar and M. Vashisth]
		{Nitesh Kumar, Tanmay Sarkar and Manmohan Vashisth}

\date{}
\address[]{Nitesh Kumar,
			Department of Mathematics, 
			Indian Institute of Technology Jammu,
			NH-44 Bypass Road, Nagrota PO, Jagti, 
			Jammu - 181 221, India.}
		\email[]{2019rma0004@iitjammu.ac.in}
\address[]{Tanmay Sarkar,
			Department of Mathematics, 
			Indian Institute of Technology Jammu,
			NH-44 Bypass Road, Nagrota PO, Jagti, 
			Jammu - 181 221, India.}
		\email[]{tanmay.sarkar@iitjammu.ac.in}
\address[]{Manmohan Vashisth,
			Department of Mathematics, 
			Indian Institute of Technology Ropar,
			Rupnagar, Punjab - 140001, India.}
		\email[]{manmohanvashisth@iitrpr.ac.in, manmohanvashisth@gmail.com}

		

		\thanks{}

\begin{document}
\begin{abstract}
We study an inverse boundary value problem associated with $p$-Laplacian which is further perturbed by a linear second order term, defined on a bounded set $\Omega$ in $\R^n, n\geq 2$. We recover the coefficients at the boundary from the boundary measurements which are given by the Dirichlet to Neumann map. Our approach relies on the appropriate asymptotic expansion of the solution and it allows one to recover the coefficients pointwise. Furthermore, by considering the localized Dirichlet-to-Neumann map around a boundary point, we provide a procedure to reconstruct the normal derivative of the coefficients at that boundary point. 
\end{abstract}
\maketitle
\section{Introduction} 
In this paper, we consider the following weighted $p$-Laplace equation which is perturbed by a linear term
\begin{equation}\label{gov_eqn}
	\begin{cases}
	\nabla\cdot (\sigma(x)\nabla u(x) + \gamma(x)|\nabla u(x)|^{p-2}\nabla u(x)) = 0, &\quad x\in~\Omega, \\
	u(x)=f(x), &\quad x\in\pl\Omega,
	\end{cases}
\end{equation}
where the domain $\Omega\subseteq\R^n,n \geq 2$ is a bounded open set with smooth boundary $\pl\Omega$. Throughout the paper, we assume that the exponent $p$ satisfies $1<p<\infty$. Furthermore, we consider the coefficients $\sigma$ and $\gamma$ are positive functions with the regularity $\sigma\in C^{\infty}(\overline{\Omega})$ and $\gamma\in L^{\infty}(\Omega)$ respectively. Moreover, the coefficients $\sigma$ and $\gamma$ are assumed to satisfy the following conditions:    
\begin{equation*}
			\begin{cases}
			0<\lambda <\sigma (x)<\lambda^{-1},\\
			0<m_1<\gamma (x)<m_1^{-1},
			\end{cases}
\end{equation*}
where $\lambda$, and $m_1$ are positive constants. The Dirichlet data $f$ is considered as follows:
$f\in W^{1/2,2}(\pl\Omega)$ whenever $p\in (1,2)$, and
$f\in W^{1-1/p,p}(\pl\Omega)$ whenever $p\in (2,\infty)$.
			
The nonlinear Dirichlet to Neumann map (DN-map) is defined by
\begin{equation}\label{dn}
	\Lambda_{\sigma,\gamma} (f) = (\sigma(x) + \gamma(x) |\nabla u |^{p-2} )\partial_\nu u |_{\partial \Omega},
\end{equation}
where $u$ is the unique solution of equation \eqref{gov_eqn} with the boundary value $f$. Throughout the paper, we assume that $\nu$ is the outer unit normal to $\partial\Omega$.
			
In case of $p=2$, the equation\eqref{gov_eqn} becomes a linear conductivity equation and there are several results related to boundary determination of conductivity coefficients appearing in the linear equation.  For instance,  in \cite{Kohn_Vogelius,sylvester1988inverse}  the smooth conductivity and all its derivatives at the boundary are recovered from the boundary measurements of the solution. The authors in \cite{Kohn_Vogelius,sylvester1988inverse} proved the uniqueness of conductivity together with its derivatives at the boundary of domain. Later 
Nachman in \cite{nac} recovered the coefficient at the boundary by assuming that $\gamma\in W^{1,m}(\Omega),~m>n$. Under the assumption $m>n/2$, Nachman in \cite{nac} also recovered the first normal derivative of $\gamma \in W^{2,m}(\Omega)$.
Furthermore, Brown \cite{brown2001recovering} described a procedure for recovering $\gamma|_{\pl\Omega}$ in a pointwise manner by assuming $\gamma\in L^{\infty}(\Omega)$.  For the conductivity equation, Nakamura and Tanuma \cite{nt, ntt} provided a formula for reconstructing the conductivity and its higher order derivatives from the localized DN-map.
Moreover, there are several local boundary determination results, for instance, one can refer to \cite{alessandrini2009local,kang2002boundary,nakamura2005numerical} and references therein.
			
Motivated by the applications in image processing, fluid mechanics and modeling of sand-piles, Calder\'{o}n-type inverse problems for the $p$-Laplace equation are studied by several authors. For instance, Salo and Zhong \cite{salo2012inverse} studied the boundary uniqueness problem associated with $p$-Laplace equation for $n\geq 2$.
%i.e., if $\Lambda_{\gamma_1}=\Lambda_{\gamma_2}$, then $\gamma_1|_{\pl\Omega}=\gamma_2|_{\pl\Omega}$. 
Brander in \cite{bt} recovered the gradient of $\gamma$ at the boundary using the Rellich's type identity. Furthermore, C\^{a}rstea and Kar in\cite{carstea2019reconstruction} considered the weighted perturbed $p$-Laplace equation and provided a procedure for reconstruction of the coefficients in the interior of the domain using the complex geometric optics solution for $n\geq 3$. However, up to our knowledge, the boundary determination problem for weighted perturbed $p$-Laplace equation concerning the coefficient $\gamma$  has not been studied in the literature. 

In this paper, we will give a procedure for recovery of the coefficients at boundary of the domain. To do this, we use the  $\epsilon$-expansion of solution which varies according to the choice of $p$. The approach of using the $\epsilon$-expansion or linearization for solving the inverse problems for nonlinear elliptic equations have been used by several authors, see for example \cite{carstea2019reconstruction,
cm,Ali_Lauri,Catalin etal,Isakov,Isakov_Nachman,Isakov_Sylvester,KN,Kian,KU-1,KU-2,Lassas_Toni_Lin_Salo}. Motivated by \cite{carstea2019reconstruction,cm,KN} our idea is to find an appropriate integral identity, from which reconstruction of the coefficient $\gamma$ at the boundary of the domain is possible. Using the approach used in \cite{cm} we derived the integral identity by constructing a sequence of explicit functions which are solutions of conductivity equation. Such solutions can be used to extract the Taylor series of the coefficient $\gamma$ at the boundary point from the knowledge of $\Lambda_{\sigma,\gamma}$.

We can now state our boundary determination results.
\begin{theorem}(Reconstruction of coefficients)\label{thm_1} \\
Suppose $p>1$ with $p\neq 2$ and $\Omega \subseteq\R^n,n\geq 2$ is a bounded domain with smooth boundary $\partial\Omega$. Furthermore, assume that the coefficients are continuous on $\overline{\Omega}$ and bounded below by a positive constant.
Then the coefficients $\sigma$ and $\gamma$ in \eqref{gov_eqn} can be reconstructed at the boundary from the knowledge of given DN-map $\Lambda_{\sigma,\gamma}$.
\end{theorem}

\begin{theorem}(Reconstruction of normal derivatives of coefficients at the boundary)\label{thm_2} \\
Assume that $p\in (1,\infty)\backslash\{2\}$ and for $n\geq 2$, $\Omega \subseteq\R^n$ is a bounded domain with smooth boundary $\partial\Omega$. Suppose $\sigma\in C^{\infty}(\overline{\Omega})$ and $D^{\alpha'}_{x'}D^{\alpha_n}_{x_n} \gamma$ is continuous around $x_0$ for any multi-index $(\alpha',\alpha_n)$ such that $|\alpha'|+2\alpha_n\leq 2$, where $x=(x',x_n),~x'\in\R^{n-1}$.
Then the normal derivative of the coefficients can be evaluated at the boundary point of $\Omega$ from the knowledge of given DN-map $\Lambda_{\sigma,\gamma}$ associated to \eqref{gov_eqn}.
\end{theorem}
			
The paper is organized as follows: we state the well-posedness results of \eqref{gov_eqn} along with corresponding $\epsilon-$expansion of the solution in section \ref{sec_2} which further includes the epsilon expansion of Dirichlet to Neumann map. In section \ref{sec_3}, we provide a detailed proof for the reconstruction of the coefficient at the boundary of the domain (cf. Theorem \ref{thm_1}). We construct a few approximate solutions using the localized Dirichlet data to extract the normal derivative of $\gamma$ and obtain a reconstruction formula in Section \ref{sec_4}.
			
\section{$\varepsilon$-expansion of the solution for perturbed $p$-Laplacian problem}\label{sec_2}
In this section, we state the existence of the solution and corresponding energy estimates for the boundary value problem (BVP) \eqref{gov_eqn}. Based on the range of $p$, we have weak solution and strong solution of the BVP \eqref{gov_eqn} corresponding to sufficiently small boundary data. We remark that such boundary data are considered since we require the resulting estimates while analyzing the boundary determination problem. More precisely, we shall use the weak solution and strong solution of BVP \ref{gov_eqn} while reconstructing the scalar coefficient $\gamma$ at the boundary.
The following Proposition \ref{prop_0} can be proved by defining suitable energy functionals and subsequently, considering the minimizing sequence (for details, we refer to \cite[Section 2]{cm}).
	
\begin{proposition}[Weak solutions]\label{prop_0}
	\begin{enumerate}[(i)]
	\item Assume that $1 < p < 2$ and $ f \in W^{1/2,2} (\partial \Omega)$. Then the BVP \eqref{gov_eqn} has a unique solution $$u \in W^{1/2,2} (\Omega) \cap W^{1,p} (\Omega), $$ and it satisfies the following estimate
		\begin{equation*}
			\|u\|_{W^{1,2} (\Omega)}  \leq  C  \|f\|_{W^{1/2,2} (\partial \Omega)},
		\end{equation*}
	where positive constant $C$ does not depend on the boundary data $f$.
\item Assume that $2 < p < \infty$ and $ f \in W^{1-{1/p},p} (\partial \Omega) $. Then the BVP \eqref{gov_eqn} has a unique solution $$u \in W^{1/2,2} (\Omega) \cap W^{1,p} (\Omega), $$ along with the stability estimate
	\begin{equation*}
		\|u\|_{W^{1,p} (\Omega)}  \leq  C  \| f \|_{W^{1-{1/p},p} (\partial \Omega)},
	\end{equation*}
	where $C=C(\Omega)>0$.
	\end{enumerate}
\end{proposition}
			
Proposition \ref{prop_0} can be further improved in case of $p>2$. In fact, there exists a strong solution for the BVP \eqref{gov_eqn}. 
By defining an appropriate operator and using the Sobolev embedding results and finally, demonstrating the fixed point of contraction mapping, the existence of the strong solution can be established. For detailed proof of the following Proposition \ref{prop_2}, we refer to \cite[Theorem 2.2]{cm}.
			
\begin{proposition}[Strong solution] \label{prop_2}
Assume that $p>2$ and $m\in\R$ such that $m>n$. Furthermore, let $f \in W^{2-{1/m},m}(\partial \Omega)$ and there holds
				\begin{align*}
				\|f\|_{W^{2-{1/m},m}(\partial \Omega)} \leq M,
				\end{align*}
				for some $M>0$.
				Then there exists a unique solution for BVP \eqref{gov_eqn}
				$$u \in W^{2,m} (\Omega).$$
				Moreover, the solution $u$ satisfies  
				\begin{equation}\label{temp_0}
				\|u\|_{W^{2,m} (\Omega)}  \leq  C  \| f \|_{W^{2-{1/m},m}(\partial \Omega)},
				\end{equation}
where constant $C >0$ is depending on the choice of $p$, domain $\Omega$ and coefficients $\sigma$, and $\gamma$.
\end{proposition} 
						
Based on the well-posedness results, described in Proposition \ref{prop_0} and Proposition \ref{prop_2}, we define the nonlinear DN-map in a weak sense as follows:
\begin{align}\label{dn1}
\langle \Lambda_{\sigma,\gamma} (f) ,w|_{\partial\Omega}\rangle = \int_{\Omega} (\sigma(x) + \gamma(x) |\nabla u |^{p-2} )\nabla u\cdot\overline{\nabla w}~dx,
\end{align}
where $w\in W^{2,m} (\Omega) \cap W^{1,2} (\Omega) $. Afterwards, we look into the $\eps$-expansion of the solution for the BVP \eqref{gov_eqn}. Since the expansion depends on the range of $p$, we carry out the analysis via two cases.
\subsection{Case-1: $p>2$}
We look for $\eps$-expansion of the solution of \eqref{gov_eqn} for $p>2$ based on the existence of strong solution mentioned in Proposition \ref{prop_2}. 
			
Let us consider the following BVP
\begin{align}\label{gov_eqn_1}
			\begin{cases}
			\nabla\cdot (\sigma(x)\nabla u_{\eps}(x) + \gamma(x)|\nabla u_{\eps}(x)|^{p-2}\nabla u_{\eps}(x)) = 0, &\quad x\in~\Omega, \\
			u_{\eps}(x)=\eps f(x), &\quad x\in\pl\Omega.\\
			\end{cases}
\end{align}
From the stability estimate \eqref{temp_0}, we obtain
			\begin{align}
			\|\eps^{-1} u_{\eps}\|_{W^{2,m}(\Omega)}\leq C \|f\|_{W^{2-1/m,m}(\partial\Omega)}.
			\end{align}
			As a consequence, we consider the following $\eps$-expansion
			\begin{align}\label{tem_2}
			u_{\eps}=\eps u_0 + \eps^{p-1} v,
			\end{align}
where $u_0 \in W^{2,m}(\Omega)$ and $v \in W^{2,m}(\Omega)$ will be determined by substituting \eqref{tem_2} in the BVP \eqref{gov_eqn_1}.
We obtain that $u_0$ satisfies the following BVP
			\begin{align}\label{eq1}
			\begin{cases}
			\nabla\cdot (\sigma(x)\nabla u_0(x))=0, & \quad x\in\Omega, \\
			u_0(x) = f(x), & \quad x\in\pl\Omega,
			\end{cases}
			\end{align}
			and eventually, $v$ satisfies the BVP
			\begin{align}\label{eq11}
			\begin{cases}
			\nabla\cdot (\sigma(x)\nabla v) + \eps^{1-p}\nabla\cdot (\gamma(x)|\nabla u_{\eps}|^{p-2}\nabla u_{\eps}) = 0, &\quad x\in~\Omega, \\
			v(x)=0 &\quad x\in\pl\Omega.\\
			\end{cases}
			\end{align}

			%%%%%%%%%
\subsection{Case-2: $1<p<2$.}
For $p\in (1,2)$, $\eps$-expansion of the solution of BVP \eqref{gov_eqn} is a consequence of existence of a weak solution of \eqref{gov_eqn}. Since $p\in(1,2)$, then we get a sublinear term instead of a nonlinear term. Therefore, for arbitrary small $\eps>0$, we will take the Dirichlet boundary data as $\eps^{-1}f$ in \eqref{gov_eqn}.
			
Let us consider
\begin{align}\label{gov_eqn_2}
			\begin{cases}
			\nabla\cdot (\sigma(x)\nabla u_\eps + \gamma(x)|\nabla u_\eps|^{p-2}\nabla u_\eps) = 0,  &\quad x\in~\Omega, \\
			u_\eps(x)=\eps^{-1}f(x), &\quad x\in\pl\Omega.\\
			\end{cases}
			\end{align}
			The well-posedness of above BVP \eqref{gov_eqn_2} is assured by the Proposition \ref{prop_0} along with the stability estimate 
			\begin{equation}\label{sol_p<1}
			\|\eps u_\eps\|_{W^{1,2} (\Omega)}  \leq  C  \|f\|_{W^{1/2,2} (\partial \Omega)}.
			\end{equation}
			By taking into account the estimate \eqref{sol_p<1}, we consider the following expansion:
			\begin{align}\label{tem_21}
			u_{\eps}=\eps^{-1} u_0 + \eps^{1-p} v,
			\end{align}
			where $u_0 \in W^{1,2}(\Omega)$ and $v \in W^{1,2}(\Omega)$ will be determined by substituting \eqref{tem_21} in the BVP \eqref{gov_eqn_2}.
			It is straightforward to observe that $u_0$ satisfies the following BVP
			\begin{align}\label{eq2}
			\begin{cases}
			\nabla\cdot (\sigma\nabla u_0)=0,&\quad x\in\Omega, \\
			u_0(x) = f(x), & \quad x\in\pl\Omega,
			\end{cases}
			\end{align}
			and $v$ satisfies the BVP
			\begin{align}\label{eq12}
			\begin{cases}
			\nabla\cdot (\sigma(x)\nabla v) + \eps^{p-1}\nabla\cdot (\gamma(x)|\nabla u_{\eps}|^{p-2}\nabla u_{\eps}) = 0, &\quad x\in~\Omega, \\
			v(x)=0, &\quad x\in\pl\Omega.\\
			\end{cases}
			\end{align}
			
\subsection{$\eps$-expansion of DN-map}		
In order to recover the parameter $\gamma$ at the boundary, we require the $\eps$-expansion of the DN-map described in \eqref{dn1}. We carry out the following analysis motivated by C{\^a}rstea et al. \cite{cm} in which the reconstruction  of $\gamma$ in the interior of domain in $\mathbb{R}^{n}$ $(n\geq 3)$ is described. 
The following estimate will be instrumental for our analysis: for $p>2$,
			\begin{align}\label{esti_01}
			\big||\eta|^{p-2}\eta-|\beta|^{p-2}\beta \big| \leq C (|\eta|+|\beta|)^{p-2} |\eta-\beta|,
			\end{align}
where $C>0$ and $\eta,\beta\in \mathbb{C}^n$ (for detailed proof, one can look into  \cite{salo2012inverse} ).		
Let us consider $w \in W^{2,m} (\Omega)\cap W^{1,2} (\Omega) $. Using the estimate \eqref{esti_01}, we obtain
\begin{equation}\label{est11}
 \begin{split}
	\Big| \int_{\Omega} & \gamma(x) (|\nabla u_\eps |^{p-2} \nabla u_\eps-\eps^{p-1}|\nabla u_0|^{p-2}\nabla u_0)\cdot\overline{\nabla w}~dx \Big| \\
	& \leq C \eps^{2p-3} \int_{\Omega} (|\eps^{-1}\nabla u_\eps|^{p-2}+|\nabla u_0|^{p-2})|\nabla v ||\nabla w|~dx = \mathcal{O}(\eps^{2p-3}).
 \end{split}
\end{equation}
With the help of expansion \eqref{tem_2} and above estimate \eqref{est11}, we get
			\begin{align*}
			\langle \Lambda_{\sigma,\gamma} (\eps f) ,w\big|_{\partial\Omega}\rangle 
			& = \int_{\Omega} (\sigma(x) + \gamma(x) |\nabla u_\eps |^{p-2} )\nabla u_\eps \cdot\overline{\nabla w}~dx \\
			& = \eps \int_{\Omega} \sigma(x)\nabla u_0 \cdot\overline{\nabla w}~dx ~+~ \eps^{p-1} \int_{\Omega} \sigma(x)\nabla v \cdot\overline{\nabla w}~dx \\
			& \qquad +\eps^{p-1}\int_{\Omega} \gamma(x) |\nabla u_0 |^{p-2} \nabla u_0 \cdot\overline{\nabla w}~dx~+~ \mathcal{O}(\eps^{2p-3}).
			\end{align*}
Subsequently, by considering the limit $\eps$ tending to zero, we have
			\begin{equation}\label{dn11}
			\begin{split}
			\lim_{\eps \rightarrow 0} \frac{1}{\eps } \langle \Lambda_{\sigma,\gamma} (\eps f) ,w\big|_{\partial\Omega}\rangle
			&= \int_{\Omega} \sigma(x)\nabla u_0 \cdot\overline{\nabla w}~dx\\
			& =\langle \Lambda_{\sigma} (f) ,w\big|_{\partial\Omega}\rangle,
			\end{split}\end{equation}
where $\Lambda_{\sigma} (f)$ is the DN-map associated to the Dirichlet BVP \eqref{eq1}.

For $p\in (1,2)$, in order to recover the parameter $\gamma$ at the boundary, we need to consider the $\eps$-expansion of the DN-map using $u_{\eps}$ described in \eqref{tem_21}.
As in the earlier case, we start with $w \in W^{1,2} (\Omega)$. Using the estimate \eqref{esti_01}, we obtain
			\begin{equation}\label{est12}
			\begin{split}
			\Big| \int_{\Omega} & \gamma(x) (|\nabla u_\eps |^{p-2} \nabla u_\eps-\eps^{1-p}|\nabla u_0|^{p-2}\nabla u_0)\cdot\overline{\nabla w}~dx \Big| \\
			& \leq C \eps^{3-2p} \int_{\Omega} (|\eps\nabla u_\eps|^{p-2}+|\nabla u_0|^{p-2})|\nabla v ||\nabla w|~dx = \mathcal{O}(\eps^{3-2p}).
			\end{split}
			\end{equation}
With the help of expansion \eqref{tem_21} and above estimate \eqref{est12}, we get
			\begin{align*}
			\langle \Lambda_{\sigma,\gamma} (\eps^{-1} f) ,w\big|_{\partial\Omega}\rangle 
			& = \int_{\Omega} (\sigma(x) + \gamma(x) |\nabla u_\eps |^{p-2} )\nabla u_\eps \cdot\overline{\nabla w}~dx \\
			& = \eps^{-1} \int_{\Omega} \sigma(x)\nabla u_0 \cdot\overline{\nabla w}~dx ~+~ \eps^{1-p} \int_{\Omega} \sigma(x)\nabla v \cdot\overline{\nabla w}~dx \\
			& \qquad +\eps^{1-p}\int_{\Omega} \gamma(x) |\nabla u_0 |^{p-2} \nabla u_0 \cdot\overline{\nabla w}~dx~+~ \mathcal{O}(\eps^{3-2p}).
			\end{align*}
By considering the DN-map $\Lambda_{\sigma} (f)$ associated with the Dirichlet BVP \eqref{eq2}, we have
			\begin{equation}\label{dn12}
			\begin{split}
			\lim_{\eps \rightarrow 0} \eps \langle \Lambda_{\sigma,\gamma} (\eps^{-1} f) ,w\big|_{\partial\Omega}\rangle
			&= \int_{\Omega} \sigma(x)\nabla u_0 \cdot\overline{\nabla w}~dx\\
			& =\langle \Lambda_{\sigma} (f) ,w\big|_{\partial\Omega}\rangle.
			\end{split}\end{equation} 
Hence we conclude that the DN-map $\Lambda_{\sigma} (f)$ is known from the knowledge of given DN-map $\Lambda_{\sigma,\gamma}$, for all values of $p$ under consideration.
			
			%%%%%%%%%%%%%%%%%%%%%%%%%%%%%%%%%%%%%%%%%%%%%%%%%%%
\section{Proof for the Theorem \ref{thm_1} }\label{sec_3}
In this section, we provide the reconstruction procedure of the coefficients $\sigma$, and $\gamma$.
In order to prove the Theorem \ref{thm_1}, we shall consider two separate cases based on the range of $p$.
\subsection{Reconstruction of $\sigma $} 
For the reconstruction of $\sigma$ from the knowledge of the DN-map $\Lambda_{\sigma} (f)$ associated to BVP \eqref{eq1}, one can refer to \cite[Theorem 5.1]{nac}.  
\subsection{Reconstruction of $\gamma$}
Using the $\eps$-expansion of the DN-map, we will give a procedure for the recovery of the parameter $\gamma$ at the boundary $\partial\Omega$. For this, our idea is to find a suitable integral identity which is known from the knowledge of the DN-map. By the asymptotic expansion of the solution along with appropriate choice of sequence of solutions of conductivity equation, the desired reconstruction of $\gamma$ at the boundary of $\Omega$ is carried out.
			
%We require the following theorem for the subsequent analysis:
%\begin{theorem}\label{t5}  {(Solutions concentrating at a boundary point)}\\
%Let $\Omega$ be any bounded open set with $C^{1}$ boundary, and let $\gamma\in C^{0}(\overline{\Omega})$ be positive. Given a point $x_0 \in \partial \Omega$, there is a sequence of solutions $\{u_M\}\subset {H}^1 (\Omega ) $ of the conductivity equation $$\nabla\cdot(\gamma\nabla u) =0 \quad in ~~\Omega $$ such that 
%\begin{equation}\label{b11}
%	\lim_{M \rightarrow \infty} \int_\Omega g|\nabla u_M|^{2} ~dx = g (x_0)
%\end{equation}
%for any $g\in C(\overline{\Omega})$.
%\end{theorem}	
%For the proof of the above theorem, one can refer to \cite[Theorem 3.5]{MIKO UHLMAN FELDMAN}.
			%%%%%%%%%%%%%%%%%%%%%%%%%%%%%%%%%%
\subsubsection{Case: $p>2$}\label{case11}
 Let us consider $w \in {W}^{2,m}(\Omega)$ such that it satisfies
\begin{align}\label{temp_w_eqn}
\nabla\cdot (\sigma \nabla w) =0 \quad \text{ on } \Omega.
\end{align}
By multiplying equation \eqref{gov_eqn_1} with $w$ and integrating the resulting equation over $\Omega$, we get
			\begin{equation*}
			\int_{\Omega} w(\nabla\cdot (\sigma\nabla u_{\eps} +\gamma|\nabla u_{\eps}|^{p-2}\nabla u_{\eps})) ~~dx = 0
			\end{equation*}
and it further yields the following equation by using the integration by parts
			\begin{equation*}
			\int_{\Omega} \nabla w \cdot (\sigma\nabla u_{\eps} +\gamma|\nabla u_{\eps}|^{p-2}\nabla u_{\eps}) ~dx = \int_{\partial\Omega} w  (\sigma + \gamma|\nabla u_{\eps}|^{p-2})\partial_\nu u_{\eps} ~dS.
			\end{equation*}
Using \eqref{temp_w_eqn}, the above expression becomes
\begin{equation}\label{int1}
\int_{\Omega} \nabla w \cdot (\gamma|\nabla u_{\eps}|^{p-2}\nabla u_{\eps}) ~dx = \int_{\partial\Omega} w  (\sigma + \gamma|\nabla u_{\eps}|^{p-2})\partial_\nu u_{\eps} ~dS
-\varepsilon\langle\Lambda_{\sigma}(f), w|_{\pl\Omega}\rangle.
\end{equation}
We observe that the right hand side terms of \eqref{int1} are known by the prescribed DN-maps $\Lambda_{\sigma,\gamma}$ and $\Lambda_{\sigma}$ respectively. As a result, the left hand side of \eqref{int1} is known to us. Let us define
			\begin{align*}
			K^{*} &:= \int_{\Omega} \nabla w \cdot  (\gamma|\nabla u_{\eps}|^{p-2}\nabla u_{\eps}) ~dx.
			\end{align*} 
The term $K^{*}$ is known from the knowledge of the DN-map associated with perturbed $p$-Laplacian problem.

We can further rewrite $K^*$ as
			\begin{align*}
			K^{*} = \int_{\Omega} \nabla w \cdot (\gamma|\nabla u_\eps|^{p-2}\nabla u_\eps) ~dx &- \eps^{p-1}~\int_{\Omega}  \nabla w \cdot(\gamma|\nabla u_0|^{p-2}\nabla u_0)~dx\\
			&+\eps^{p-1} \int_{\Omega} \nabla w \cdot(\gamma|\nabla u_0|^{p-2}\nabla u_0) ~dx.
			\end{align*}
By taking into account the estimate \eqref{est11}, we deduce that
			\begin{align*}
			K^{*}& = \eps^{p-1} \int_{\Omega} \nabla w\cdot(\gamma|\nabla u_0|^{p-2}\nabla u_0) ~dx + \mathcal{O}(\eps^{2p-3}).
			\end{align*}
Further, we define $I(u_0,w)$ as follows:
			\begin{align}\label{defn_I}
			I (u_0,w) := \int_{\Omega}\nabla w \cdot(\gamma|\nabla u_0|^{p-2}\nabla u_0) ~dx.
			\end{align}
As a consequence, we have
			\begin{align}
			K^{*}= \eps^{p-1} I(u_0,w) + \mathcal{O}(\eps^{2p-3}),
			\end{align}
where $w\in W^{2,m}(\Omega)$ and $u_0$ satisfies the BVP \eqref{eq1}. Furthermore, we observe that 
\begin{align*}
I (u_0 , w) = \lim_{\eps \rightarrow 0^{+}} \eps^{1-p} \Big(K^{*}-\mathcal{O} (\eps^{2p-3})\Big)
\end{align*}
and subsequently, $I(u_0,w)$  is known to us.
			
Let us choose the functions $u_1,u_2,u_3 \in {C}^\infty (\Omega)$ such that 
\begin{align}\label{cdt}
			\nabla\cdot(\sigma\nabla u_j) =0, \quad j = 1,2,3.
\end{align}
In addition, we assume that $u_1$ is real valued and $\nabla u_1$ does not have any zeros.
Since $I(u_0,w)$ is known for any $w\in W^{2,m}(\Omega)$. In particular $I(u_0,u_3)$ is known. 
By defining $J_1$ and $J_2$ as follows:
			\begin{align*}
			J_1(u_1, u_2, u_3) :&= \frac{2}{p-2} \partial_{\bar z}[ I (u_{1}+z \overline {u_2} , u_3)]\Big|_{z=0}\\
			&= \int_\Omega \gamma|\nabla u_1|^{p-4}(\nabla u_1\cdot \nabla u_2 ) (\nabla u_1 \cdot \nabla u_3)   ~dx, \\
			J_2(u_1 , u_2,u_3) :&= \frac{2}{p-2} \partial_{z}[ I (u_{1}+z u_2 , u_3)]\Big|_{z=0} \\
			&=\int_\Omega \gamma |\nabla u_1|^{p-4}[(\overline{\nabla u_1}\cdot \nabla u_2 ) (\nabla u_1 \cdot \nabla u_3)+\frac{2}{p-2}|\nabla u_1|^{2}(\nabla u_2\cdot \nabla u_3 )]~dx,
			\end{align*}
we conclude that the functionals $J_1(u_1 , u_2,u_3)$ and $J_2(u_1 , u_2,u_3)$ will be known to us.
Furthermore, we define
			\begin{align}\label{b}
			& \beta_{u_1}(x) := \gamma|\nabla u_1(x)|^{p-2}
			\end{align}
			and
			\begin{align*}
			& K (u_1 , u_2,u_3) :=\frac{p-2}{2} [J_2(u_1 , u_2,u_3) - J_1(u_1 , u_2,u_3)].
			\end{align*}
			Hence $K$ can be represented as
			\begin{equation*}
			K (u_1, u_2,u_3) = \int_\Omega \beta_{u_1} (\nabla u_2\cdot \nabla u_3 ) ~dx.
			\end{equation*}
For each choice of $u_2$ and $u_3$ satisfying \eqref{cdt}, the term $ K (u_1, u_2,u_3)$ is known. Since, $\sigma$ is real valued, therefore by choosing $u_3=\overline{u_2}$, we find that
\begin{equation}\label{k}
	K (u_1, u_2,\overline{u_2}) = \int_\Omega \beta_{u_1} \nabla u_2\cdot\overline{\nabla u_2} ~dx,
\end{equation} 
is known to us.  
			
Let $x_0\in \partial\Omega$ be arbitrary, but fixed. From the knowledge of $ K (u_1, u_2,\overline{u_2})$ in \eqref{k}, we will recover $\beta_{u_1}$ at $x_0$. By using \eqref{b}, we recover $\gamma(x_0)$ from $\beta_{u_1}(x_0)$. Following the similar approach presented in \cite{salo2012inverse, brown2001recovering}, we can construct a sequence of solutions $\{u_M\}$ satisfying the conductivity equation \eqref{cdt} with coefficient frozen at the boundary point $x_0$. Suppose $\rho$ is a $C^1$ function and we choose $\zeta$  as a unit tangent vector to $\partial\Omega$ at $x_0$ such that \begin{align}\label{temp_1}
			\zeta \cdot \nabla\rho(x_0)=0,\qquad |\zeta|=|\nabla\rho(x_0)|.
\end{align}
Furthermore, we choose a cut-off function $\eta \in C^{\infty}(\mathbb{R}^n)$ with $0\leq \eta (x-x_0) \leq 1$ such that 
			\begin{align*}
			\eta (x-x_0) = 
			\begin{cases}
			1, & \quad |x-x_0| \leq \frac{1}{2},\\
			0, & \quad |x-x_0| \geq 1.
			\end{cases}
			\end{align*}
For large positive numbers $M$ and $N$, we define $\eta_M(x)=\eta(Mx)$ and $N=N(M)$ with $\quad \frac{M}{N}=\smallO (1)$ as $M\rightarrow \infty$. Subsequently, we can deduce the following:
 \begin{align}\label{temp:conv}
 M^{n-1} N \int_{\Omega} \eta(Mx)\exp({-2N\rho(x)})~dx \longrightarrow \frac{1}{2}\int_{\R^{n-1}} \eta (x',x_0)~dx' \quad \text{ as } M\rightarrow \infty
 \end{align}
since $\eta$ is supported in the ball $B(x_0,1)$.

For each $M \in \mathbb{N}$ and $M\geq 2$, we define a solution of \eqref{cdt} as 
\begin{align}\label{temp: recon_1}
  u_M := C_{M,N} (w_0 + w_1),
\end{align}
where $C_{M,N}$ is a scaling constant and it is described by $$C_{M,N}=\sqrt{\frac{M^{n-1}N^{-1}}{c_\eta}}$$ 
			in which $c_{\eta}$ is given by   
			$$\qquad c_\eta := \frac{1}{2}\int_{\R^{n-1}} \eta (x',x_0)~dx'.$$ 
In \eqref{temp: recon_1}, we consider $w_0$ as an approximate solution of \eqref{cdt} and it assumes the following form
			$$ w_0(x)= \eta (M(x-x_0)) \exp\big({N(i \zeta \cdot x-\rho (x))}\big), \qquad x\in \mathbb{R}^n.$$ 
The above form of $w_0$ is justified thanks to \eqref{temp_1}.
As a consequence, we observe that $w_1$ satisfies 
			\begin{align*}
			\begin{cases}
			\nabla\cdot (\sigma(x)\nabla w_1(x))=-\nabla\cdot (\sigma(x)\nabla w_0(x)), & \quad x\in\Omega, \\
			w_1(x) =0, & \quad x\in\pl\Omega,
			\end{cases}
			\end{align*}
in addition with \cite{salo2012inverse}
	\begin{align}\label{int1234}
		\int_{\Omega}|\nabla w_0|^2~dx=\mathcal{O} (M^{1-n}N), \qquad \int_{\Omega}|\nabla w_1|^2~dx=\smallO (M^{1-n}N) \quad \text{ as } M\rightarrow \infty.
	\end{align}\\
Note that $w_0\in C^1_c(\mathbb{R}^n)$ is vanishing outside a small ball $B(x_0,1/M)$. For this choice of sequence $\{u_M\}$ in \eqref{temp: recon_1} we get
			\begin{align}\label{it123}
			\int_\Omega \beta_{u_1}\nabla u_M\cdot\overline{\nabla u_M}~dx=\frac{M^{n-1}N^{-1}}{c_\eta}\int_\Omega \beta_{u_1} (|\nabla w_0|^2+\nabla w_0\cdot\overline{\nabla w_1}+\overline{\nabla w_0}\cdot\nabla w_1+|\nabla w_1|^2)~dx.
			\end{align} 
We observe that 
\begin{align*}
\bigg| \int_\Omega \beta_{u_1} \big(\nabla w_0\cdot\overline{\nabla w_1} & +\overline{\nabla w_0}\cdot\nabla w_1+|\nabla w_1|^2\big)~dx\bigg| \\
& \leq C\big(\|\nabla w_0\|_{L^2(\Omega)} + \|\nabla w_1\|_{L^2(\Omega)}\big)\|\nabla w_1\|_{L^2(\Omega)} = \smallO (M^{1-n}N),
\end{align*}
where we have used the Cauchy-Schwarz inequality and estimates \eqref{int1234}.

For the first term on right hand side of \eqref{it123}, we have
 \begin{align*}
 \nabla w_0 & = N(i\zeta - \nabla\rho)\eta_M \exp\big({N(i\zeta\cdot x - \rho(x))}\big) + M(\nabla\eta)(M\cdot)\exp\big({N(i\zeta\cdot x - \rho(x))}\big)\\
 & = : T_1 + T_2,
 \end{align*}
 and consequently,
 \begin{align*}
 \int_{\Omega} \beta_{u_1}(x)|\nabla w_0|^2 ~dx &= \beta_{u_1}(x_0) \int_{\Omega} |\nabla w_0|^2 ~dx + \int_{\Omega} \big(\beta_{u_1}(x)-\beta_{u_1}(x_0) \big)|\nabla w_0|^2 ~dx\\
 & = \beta_{u_1}(x_0) \int_{\Omega} |T_1|^2~dx + \beta_{u_1}(x_0) \int_{\Omega}(T_1\cdot \overline{T_2}+  \overline{T_1}\cdot T_2 + |T_2|^2)~dx\\
 & \qquad+ \int_{\Omega} \big(\beta_{u_1}(x)-\beta_{u_1}(x_0) \big)|\nabla w_0|^2 ~dx\\
 & = \beta_{u_1}(x_0) \int_{\Omega} |T_1|^2~dx +\smallO(M^{1-n}N),
 \end{align*}
 where we have used the fact that $\|T_1\|^2_{L^2(\Omega)} = \mathcal{O}(M^{1-n}N)$, and $\|T_2\|^2_{L^2(\Omega)} = \smallO(M^{1-n}N)$ as $M\rightarrow \infty$, and the continuity of $\beta_{u_1}$ at $x_0$.
 Moreover, we have
 \begin{align*}
 M^{n-1}N^{-1}\int_{\Omega} |T_1|^2~dx & = M^{1-n}N \int_{\Omega} \eta(M(x-x_0))^2 \exp({-2N\rho})(1+|\nabla\rho|^2)~dx\\
 &= M^{1-n}N \bigg(2\int_{\Omega} \eta(M(x-x_0))^2 \exp({-2n\rho})~dx \\
 &\qquad + \int_{\Omega} \eta(M(x-x_0))^2 \exp({-2N\rho})(|\nabla\rho|^2-|\nabla \rho(x_0)|^2)~dx\bigg),
 \end{align*}
 where the first term and second term in the right hand side converges to $c_{\eta}$ (due to \eqref{temp:conv}) and $\smallO(1)$ (since $x\mapsto |\nabla\rho|^2$ is continuous at $x_0$) respectively as $M\rightarrow \infty$. Hence we obtain
\begin{align}\label{int12345}
	 \frac{M^{n-1}N^{-1}}{c_\eta} \int_\Omega \beta_{u_1}(x) |\nabla w_0|^2 ~dx  \longrightarrow \beta_{u_1}(x_0) \quad \text{ as } M\rightarrow\infty.
\end{align}
As a consequence, we conclude that from \eqref{it123} 
	\begin{align*}
		\int_\Omega \beta_{u_1}(x)\nabla u_M\cdot\overline{\nabla u_M}~dx \longrightarrow \beta_{u_1}(x_0) \quad \text{ as } M\rightarrow \infty.
	\end{align*}
			
In \eqref{k} replacing $u_2 $ by $u_M $, we are able to find $\beta_{u_1}(x_0)$. Hence $\gamma(x_0)$ is known from \eqref{b}. Since $x_0\in\partial\Omega$ is arbitrary, one can recover $\gamma$ at the boundary of the domain.
			
\subsubsection{Case: $1<p<2$}
Let $w$ be any function such that $w \in {W}^{1,2}(\Omega)$ and it satisfies
\begin{align}\label{temp:eqn_w}
\nabla\cdot (\sigma \nabla w) = 0 \quad \text{ in }\Omega.
\end{align}
We multiply the equation \eqref{gov_eqn_2} by $w$ and integrate the resulting equation over $\Omega$ to obtain
			\begin{equation}\label{01}
			\int_{\Omega} w \Big(\nabla\cdot (\sigma\nabla u_{\eps} +\gamma|\nabla u_{\eps}|^{p-2}\nabla u_{\eps})\Big) ~~dx = 0.
			\end{equation}
The above equation reduces to 
			\begin{equation*}
			\int_{\Omega} \nabla w \cdot (\sigma\nabla u_{\eps} +\gamma|\nabla u_{\eps}|^{p-2}\nabla u_{\eps}) ~dx = \int_{\partial\Omega} w  (\sigma + \gamma|\nabla u_{\eps}|^{p-2})\partial_\nu u_{\eps} ~dS,
			\end{equation*}
			where we have applied integration by parts. With the help of \eqref{temp:eqn_w} we get
\begin{equation}\label{int12}
\int_{\Omega} \nabla w \cdot (\gamma|\nabla u_{\eps}|^{p-2}\nabla u_{\eps}) ~dx = \int_{\partial\Omega} w  (\sigma + \gamma|\nabla u_{\eps}|^{p-2})\partial_\nu u_{\eps} ~dS
- \varepsilon^{-1}\langle\Lambda_{\sigma}(f),w|_{\pl\Omega} \rangle.
\end{equation}			
Since the DN-maps $\Lambda_{\sigma,\gamma,p}$ and $\Lambda_{\sigma}$ defined in a weak sense in \eqref{dn1} are given to us, therefore the right hand side of \eqref{int12} is also known.
			
Let us define
\begin{align*}
	L^{*} &:= \int_{\Omega} \nabla w \cdot  (\sigma\nabla u_{\eps} +\gamma|\nabla u_{\eps}|^{p-2}\nabla u_{\eps}) ~dx.
\end{align*} 
From the knowledge of the DN-map associated with perturbed problem, we know the integral $L^{*}$.
Moreover, $L^{*}$ can be represented as follows
			\begin{align*}
			L^{*} & = \int_{\Omega} \nabla w \cdot (\gamma|\nabla u_\eps|^{p-2}\nabla u_\eps) ~dx - \eps^{1-p}~\int_{\Omega}  \nabla w \cdot(\gamma|\nabla u_0|^{p-2}\nabla u_0)~dx\\
			&\qquad \qquad\qquad +\eps^{1-p} \int_{\Omega} \nabla w \cdot(\gamma|\nabla u_0|^{p-2}\nabla u_0) ~dx\\
			& = \eps^{1-p} \int_{\Omega} \nabla w \cdot(\gamma|\nabla u_0|^{p-2}\nabla u_0) ~dx + \mathcal{O}(\eps^{3-2p})\\
			& = \eps^{1-p} I(u_0,w) + \mathcal{O}(\eps^{3-2p}),
			\end{align*}
where the estimate \eqref{est12} is used and $I(u_0,w)$ is defined in \eqref{defn_I}. Hence 
			\begin{align} \label{tmp_1}
			I (u_0,w) = \int_{\Omega}\nabla w \cdot(\gamma|\nabla u_0|^{p-2}\nabla u_0) ~dx
			\end{align}
is known to us.
As we have described a reconstruction procedure of the coefficient at the boundary using $I(u_0,w)$ in case of $p>2$, in a similar way, we can recover the coefficient $\gamma$ at the boundary $\partial\Omega$ from \eqref{tmp_1} for $1<p<2$.
%%%%%%%%%%%%%%%%%%%%%%%%%%%%%%%%%%%%%%%%%%%%%%%%%
 \section{Proof for the Theorem  \ref{thm_2}} \label{sec_4}
In this section, we present a procedure for the reconstruction of normal derivative of $\gamma$ at the boundary point $x_0$ of $\Omega$ under certain regularity assumption on $\gamma$. Unlike in \cite{nac,ntt}, we do not require the values of $\gamma$ in the neighbourhood of $x_0$. However, we require certain regularity assumption on $\gamma$ around $x_0$. We also use the fact that $\gamma$ is known at $x_0$ as described in section \ref{sec_3}.
\subsection{Reconstruction of normal derivative of $\sigma$}
Let $\alpha\in\mathbb{R}^n\backslash\{0\}$. The reconstruction of $\partial_{\alpha}\sigma$ from the knowledge of the DN-map $\Lambda_\sigma(f)$ associated to BVP \eqref{eq1}, can be done using \cite{nt, ntt}.
\subsection{Reconstruction of normal derivative of $\gamma$}
For every $x\in \mathbb{R}^n$, we write $x=(x',x_n)$ with $x'\in \mathbb{R}^{n-1}$.
Without loss of generality, we assume that $x_0=0$ and $\pl\Omega$ is flat around $x_0$. More precisely, there exists $\delta>0$ such that
\begin{align}\label{temp_local_coord}
\Omega \cap B_{\delta}(x_0) = \{x\in B_{\delta}(x_0): x_n > 0\}.
\end{align}
In fact, using the boundary normal coordinates \cite{lee1989determining}, we can locally transform the coefficients to the form \eqref{temp_local_coord}. Let $\zeta=(\zeta',0)$ be an unit tangent vector to $\pl\Omega$ at $x_0$.

Motivated by \cite{salo2012inverse,ntt}, we wish to choose the approximate solutions $\Phi_N$ and $\Psi_N$ having oscillating boundary data supported in a neighbourhood of $x_0$. For $N\in\N$ and $x\in \pl\Omega$, we consider 
\begin{equation}\label{defn_phi_psi}
\begin{split}
\Phi_N(x) &:= \exp\big({iN x' \cdot \zeta'-N x_n}\big)\eta(\sqrt{N}x'),\\
\Psi_N(x)&:= \exp\Big({i\frac{N}{2}x' \cdot \zeta'-\frac{N}{2}x_n} \Big)\eta(\sqrt{N}x');
\end{split}\end{equation}
corresponding to the boundary data
\begin{align}
    f_N(x') & = \exp\big({iN x' \cdot \zeta'}\big)\eta(\sqrt{N}x'), \qquad 
    g_N(x') = \exp\big({i\frac{N}{2} x' \cdot \zeta'}\big)\eta(\sqrt{N}x');
\end{align}
respectively, where $\eta\in C^2$ is a cutoff function supported in a unit ball such that
\begin{equation}\label{defn_eta}
0\leq \eta(x')\leq 1, \quad \int_{\mathbb{R}^{n-1}} \eta^2(x')~dx'=1.  
\end{equation}
We have observed that (cf. \eqref{k})
\begin{align*}
K(u_1,u_2,u_3) = \int \beta_{u_1}(x)\nabla u_2\cdot\nabla u_3~dx
\end{align*}
is known for any $u_1, u_2$ and $u_3$ satisfying the conductivity equation \eqref{cdt}. Let us denote $h:=\beta_{u_1}$.
To prove the Theorem \ref{thm_2}, we use the following proposition:
%%%%%%%%%%%%%%%%%%%%%%%%%%%%
\begin{proposition}\label{prop_1}
	Let $\pl\Omega$ be locally $C^2$ at $x_0$. Suppose $D^{\alpha'}_{x'}D^{\alpha_n}_{x_n} \gamma$ is continuous around $x_0$ for any multi-index $(\alpha',\alpha_n)$ such that $|\alpha'|+2\alpha_n\leq 2$. Then
	\begin{align*}
	\lim_{N\rightarrow\infty} N^{\frac{n-1}{2}} \Big(2 K(u_1,u_2,\overline{u_2}) - K(u_1,u_3,\overline{u_3})\Big) & = \frac{1}{2}{\pl}_{x_n}h(x_0) + \frac{3}{2} h(x_0)\int_{\R^{n-1}}|\nabla_{x'}\eta|^2~dx'.
	\end{align*}
\end{proposition}
%%%%%%%%%%%%%%%%%%%%%%%%
\begin{proof}
Let $\xi\in C^\infty([0,\infty))$ be a cutoff function which satisfies $0\leq \xi\leq 1$ with
\begin{align*}
\xi(x_n) = 
\begin{cases}
1, & \quad 0\leq x_n\leq \frac{1}{2},\\
0, & \quad x_n \geq 1.
\end{cases}
\end{align*}
For each $ N\in \mathbb{N}$, we define $\xi_N(x_n):=\xi(\sqrt{N}x_n)$. Our computation is heavily relying on the choice of approximate solutions of conductivity equation \eqref{cdt}. As a consequence, we consider
\begin{equation}\label{sol}
u_2=\xi_N\Psi_N + s_N, \qquad u_3=\xi_N\Phi_N + r_N,
\end{equation}
where $\Phi_N$ and $\Psi_N$ are defined in \eqref{defn_phi_psi}, and $s_N$ satisfies
\begin{equation}\label{sn}
\begin{cases}
\nabla\cdot (\sigma(x)\nabla s_N)=-\nabla\cdot \big(\sigma(x)\nabla (\xi_N\Psi_N)\big), & \quad x\in\Omega, \\
s_N =0, & \quad x\in\pl\Omega,
\end{cases}
\end{equation}
and similarly, $r_N$ satisfies
\begin{equation}\label{rn}
\begin{cases}
\nabla\cdot (\sigma(x)\nabla r_N)=-\nabla\cdot \big(\sigma(x)\nabla (\xi_N\Phi_N)\big), & \quad x\in\Omega, \\
r_N =0, & \quad x\in\pl\Omega.
\end{cases}
\end{equation}
Taking into account \eqref{sol}, we have
\begin{align*}
2 & K (u_1,u_2,\overline{u_2})- K (u_1,u_3,\overline{u_3})\\
%&=4K (u_1,\xi_N\psi_N + s_N,\overline{\xi_N\psi_N + s_N})-2K (u_1,\xi_N\phi_N + r_N,\overline{\xi_N\phi_N + r_N})\\
&=\int_\Omega h(x) \bigg\{2\nabla (\xi_N\Psi_N + s_N)\cdot\overline{\nabla (\xi_N\Psi_N + s_N)}-\nabla (\xi_N\Phi_N + r_N)\cdot\overline{\nabla (\xi_N\Phi_N + r_N)}\bigg\} ~dx\\
&=\int_\Omega h(x) \bigg\{2\nabla (\xi_N\Psi_N)\cdot\overline{\nabla (\xi_N\Psi_N)}-\nabla (\xi_N\Phi_N)\cdot\overline{\nabla (\xi_N\Phi_N)}\bigg\} ~dx \\
&\quad  + 2\int_\Omega h(x) \nabla s_N \cdot\overline{\nabla (\xi_N\Psi_N)} ~dx
 - \int_\Omega h(x) \nabla r_N \cdot\overline{\nabla (\xi_N\Phi_N)}  ~dx \\
& \quad +2\int_\Omega h(x) \nabla (\xi_N\Psi_N) \cdot\overline{\nabla s_N } ~dx
 -\int_\Omega h(x) \nabla (\xi_N\Psi_N) \cdot\overline{\nabla r_N } ~dx\\
&\quad + 2\int_\Omega h(x) \nabla s_N \cdot\overline{\nabla s_N } ~dx 
-\int_\Omega h(x) \nabla r_N \cdot\overline{\nabla r_N } ~dx\\
& =: T_1+T_2+T_3+T_4+T_5+T_6+T_7,
\end{align*}
where $T_2$ is given by
\begin{align*}
T_2=2\int_\Omega h(x) \nabla s_N \cdot\overline{\nabla (\xi_N\Psi_N)} ~dx,
\end{align*}
and using the arguments from \cite{nt,ntt} we have $T_2 = \smallO(N^{-\frac{n-1}{2}})$ as $N\rightarrow\infty$. Following the similar argument in \cite{kang2002boundary}, we can demonstrate that $N\rightarrow\infty$,
\begin{align*}
T_3 &=- \int_\Omega h(x) \nabla r_N \cdot\overline{\nabla (\xi_N\Phi_N)}  ~dx= \smallO(N^{-\frac{n-1}{2}}),\\
T_4 &=2\int_\Omega h(x) \nabla (\xi_N\Psi_N) \cdot\overline{\nabla s_N } ~dx= \smallO(N^{-\frac{n-1}{2}}),\\
T_5 &=-\int_\Omega h(x) \nabla (\xi_N\Psi_N) \cdot\overline{\nabla r_N } ~dx= \smallO(N^{-\frac{n-1}{2}}),
\end{align*}
and using the equation \eqref{sn}-\eqref{rn} and estimates given for $T_2$ and $T_3$ one can show that as $N\rightarrow\infty$,
\begin{align*}
T_6&=2\int_\Omega h(x) \nabla s_N \cdot\overline{\nabla s_N } ~dx= \smallO(N^{-\frac{n-1}{2}}),\\
T_7&=-\int_\Omega h(x) \nabla r_N \cdot\overline{\nabla r_N } ~dx= \smallO(N^{-\frac{n-1}{2}}).
\end{align*}
We set the scaling transformation 
\begin{equation}\label{trns}
y_i=\sqrt{N}x_i, \quad i=1,2,\cdots,n-1,\qquad y_n= Nx_n.
\end{equation}
Taking into account the supports of $\xi_N\Psi_N (y)$ and $\xi_N\Phi_N (y)$, we define
\begin{align*}
 \Omega_N=\Big \{x:|x'|\leq \frac{1}{\sqrt{N}}, 0\leq x_n \leq \frac{1}{2\sqrt{N}}\Big\},\quad
 \Omega'_N=\Big \{x:|x'|\leq \frac{1}{\sqrt{N}}, \frac{1}{2\sqrt{N}}\leq x_n\leq \frac{1}{\sqrt{N}}\Big\}
\end{align*}
and hence, 
we can rewrite $T_1$ as 
\begin{align*}
T_1 & =\int_{\Omega_N} h(x) \bigg\{2\nabla \Psi_N\cdot\overline{\nabla \Psi_N}
- \nabla \Phi_N\cdot\overline{\nabla \Phi_N}\bigg\}~dx\\
&\quad +\int_{\Omega_N'} h(x) \bigg\{2\nabla (\xi_N\Psi_N)\cdot\overline{\nabla (\xi_N\Psi_N)} - \nabla (\xi_N\Phi_N)\cdot\overline{\nabla (\xi_N\Phi_N)}\bigg\} ~dx\\
&= T_{11} + T_{12},
\end{align*}
where $T_{11}$ and $T_{12}$ are given by
\begin{align*}
&T_{11} :=\int_{\Omega_N} h(x) \bigg\{2\nabla \Psi_N\cdot\overline{\nabla \Psi_N}-\nabla \phi_N\cdot\overline{\nabla \Phi_N}\bigg\}~dx,\\
&T_{12}:=\int_{\Omega'_N} h(x) \bigg\{2\nabla (\xi_N\Psi_N)\cdot\overline{\nabla (\xi_N\Psi_N)}-\nabla (\xi_N\Phi_N)\cdot\overline{\nabla (\xi_N\Phi_N)}\bigg\} ~dx.
\end{align*}
$T_{11}$ can be further represented by
\begin{align*}
T_{11}&=\int_{\Omega_N} h(x) \bigg\{2\nabla \Psi_N\cdot\overline{\nabla \Psi_N} - \nabla \Phi_N\cdot\overline{\nabla \Phi_N}\bigg\}~dx\\
&=N^2\int_{\Omega_N} h(x',x_n) \eta^2(\sqrt{N}x')~\big(\exp({-Nx_n})-2\exp({-2Nx_n})\big)~dx\\
&\qquad +N \int_{\Omega_N} h(x',x_n) \vert(\nabla_{y'}\eta)(\sqrt{N}x')\vert^2~\big(2\exp({-Nx_n})- \exp({-2Nx_n})\big)~dx,
\end{align*}
and using the scaling transformation introduced in \eqref{trns}, we get
\begin{equation}\label{temp_deri_1}
\begin{split}
N^{\frac{n-1}{2}} T_{11} &=N\int^{\sqrt{N}/2}_{0} \int_{|y'|\leq 1} h\left(\frac{y'}{\sqrt{N}},\frac{y_n}{N}\right) \eta^2(y')~\big(\exp({-y_n})-2\exp({-2y_n})\big)~dy'dy_n\\
&\quad +\int^{\sqrt{N}/2}_{0} \int_{|y'|\leq 1} h\left(\frac{y'}{\sqrt{N}},\frac{y_n}{N} \right) |(\nabla_{y'}\eta)(y')|^2~\big(2\exp({-y_n})-\exp({-2y_n})\big)~~dy'dy_n.
\end{split}\end{equation}
From the regularity assumption on $h$, we can expand $h$ around $x_0=0$. More precisely, we can write
\begin{equation}\label{i1}
\begin{split}
h\left(\frac{y'}{\sqrt{N}},\frac{y_n}{N}\right)=&h(x_0) +\frac{1}{\sqrt{N}}\sum_{|\alpha'|=1}{\pl}^{|\alpha'|}_{x'^{\alpha'}}h(x_0)y'^{\alpha'} +\frac{1}{N}\pl_{x_n} h(x_0)y_n\\
&+\frac{1}{N}\sum_{|\alpha'|=2}\frac{1}{\alpha !} {\partial}^{|\alpha'|}_{x'^{\alpha'}}h(x_0)y'^{\alpha'}  +\smallO(N^{-1}) \qquad \text{as } N\rightarrow\infty.
\end{split}
\end{equation}
It is observed that as $N\rightarrow\infty$,
\begin{align*}
&\int^{\sqrt{N}/2}_{0}(\exp{-y_n}-2\exp{-2y_n})~dy_n=\mathcal{O}\big(\exp({-\sqrt{N}/{2}})\big),\\
&\int^{\sqrt{N}/2}_{0}\big(2\exp({-y_n})-\exp({-2y_n})\big)~dy_n=\frac{3}{2}
+\mathcal{O}\big(\exp({-\sqrt{N}/{2}})\big),\\
&\int^{\sqrt{N}/2}_{0}y_n \big(\exp({-y_n})-2\exp({-2y_n})\big)~dy_n=\frac{1}{2}
+\mathcal{O}(\sqrt{N}e^{-\sqrt{N}/{2}}).
\end{align*}
With the help of above estimates and expansion \eqref{i1}, and using the properties of $\eta$ introduced in \eqref{defn_eta}, we end up with
\begin{equation}\label{esti_1}
\begin{split}
N^{\frac{n-1}{2}}&\int_{\Omega_N} h(x) \bigg\{2\nabla \Psi_N\cdot\overline{\nabla \Psi_N}
- \nabla \Phi_N\cdot\overline{\nabla \Phi_N}\bigg\}~dx\\
&=\frac{1}{2}\pl_{x_n} h(x_0) + \frac{3}{2} h(x_0) \int_{\mathbb{R}^{n-1}}\vert(\nabla_{x'}\eta)\vert^2~dx'+\smallO(1) \qquad \text{as }N\rightarrow \infty.
\end{split}
\end{equation}
We see that
\begin{align*}
    \nabla(\xi_N\Psi_N) = \exp\left(i\frac{N}{2}x'\cdot\zeta' - \frac{N}{2}x_n\right)\begin{bmatrix}
    \Big\{i\frac{N}{2}\zeta' \eta(\sqrt{N}x')+\sqrt{N}(\nabla_{y'}\eta)(\sqrt{N}x')\Big\}\xi(\sqrt{N}x_n)\\
    \Big\{-\frac{N}{2}\xi(\sqrt{N}x_n) + \sqrt{N} (\nabla_{x_n}\xi)(\sqrt{N}x_n)\Big\}\eta(\sqrt{N}x')
    \end{bmatrix},
\end{align*}
and
\begin{align*}
    \nabla(\xi_N\Phi_N) = \exp\left(iN x'\cdot\zeta' - N x_n\right)\begin{bmatrix}
    \Big\{iN\zeta' \eta(\sqrt{N}x') + \sqrt{N}(\nabla_{y'}\eta)(\sqrt{N}x')\Big\}\xi(\sqrt{N}x_n)\\
    \Big\{-N\xi(\sqrt{N}x_n) + \sqrt{N} (\nabla_{x_n}\xi)(\sqrt{N}x_n)\Big\}\eta(\sqrt{N}x')
    \end{bmatrix}.
\end{align*}
As a consequence, we have
\begin{align*}
T_{12}& = \int_{\Omega'_N} h(x)\bigg\{2\nabla (\xi_N\Psi_N)\cdot\overline{\nabla (\xi_N\Psi_N)}-\nabla (\xi_N\Phi_N)\cdot\overline{\nabla (\xi_N\Phi_N)}\bigg\} ~dx\\
& = {N^2}\int_{\Omega'_N} h(x',x_n)\xi^2(\sqrt{N}x_n)\eta^2(\sqrt{N}x') 
\{\exp(-Nx_n) - 2\exp(-2Nx_n)\}~dx\\
& + N \int_{\Omega'_N} h(x',x_n)\xi^2(\sqrt{N}x_n)|\nabla_{y'}\eta|^2(\sqrt{N}x')\{2\exp(-Nx_n) - \exp(-2Nx_n)\}~dx\\
& -2N^{3/2} \int_{\Omega'_N} h(x)\xi(\sqrt{N}x_n)\eta(\sqrt{N}x')(\nabla_{x_n}\xi)(\sqrt{N}x_n)\{\exp(-Nx_n) - \exp(-2Nx_n)\}~dx\\
& + N \int_{\Omega'_N} h(x',x_n) \eta^2(\sqrt{N}x') |(\nabla_{x_n}\xi)(\sqrt{N}x_n)|^2 \{2\exp(-Nx_n) - \exp(-2Nx_n)\}~dx\\
& =: I + II + III + IV.
\end{align*}
Note that
\begin{align*}
I &= {N^2}\int_{\Omega'_N} h(x',x_n)\xi^2(\sqrt{N}x_n)\eta^2(\sqrt{N}x') 
\{\exp(-Nx_n) - 2\exp(-2Nx_n)\}~dx\\
& = \frac{N}{N^{\frac{n-1}{2}}} \int_{\frac{\sqrt{N}}{2}}^{\sqrt{N}}
\int_{|y'|\leq 1} h\left(\frac{y'}{\sqrt{N}},\frac{y_n}{N}\right)\xi^2\left(\frac{y_n}{\sqrt{N}} \right) \eta^2(y') \{\exp(-y_n) - 2\exp(-2y_n)\}~dy'dy_n,
\end{align*}
where we have used the scaling transformation \eqref{trns}. Taking into account
the inequality
\begin{align*}
\int_{\frac{\sqrt{N}}{2}}^{\sqrt{N}} \{\exp(-y_n) - 2\exp(-2y_n)\}~dy_n
= \mathcal{O}\big(\exp(-\sqrt{N}/2)\big) \quad \text{as } N \rightarrow\infty,
\end{align*}
we can deduce that
\begin{align*}
    |I| \leq C\exp(-\sqrt{N}/2)\quad \text{ as } N \rightarrow\infty,
\end{align*}
for some constant $C$. Following the similar argument as above, we see that
\begin{align*}
    |II| + |III| + |IV| \leq C\exp(-\sqrt{N}/2)\quad \text{as } N \rightarrow\infty.
\end{align*}
Subsequently, we can conclude that
\begin{align}\label{esti_2}
    T_{12} =\mathcal{O}\big(\exp({-\sqrt{N}/{2}})\big) \quad \text{as } N\rightarrow\infty.
\end{align}
Hence combining the above estimates \eqref{esti_1} and \eqref{esti_2}, \eqref{temp_deri_1} reduces to
\begin{equation}\label{final}
\begin{split}
\lim_{N\rightarrow\infty} N^{\frac{n-1}{2}}\big(2K (u_1,u_2,\overline{u_2})
& - K (u_1,u_3,\overline{u_3})\big)\\
& =\frac{1}{2}\pl_{x_n}h(x_0)
+ \frac{3}{2} h(x_0) \int_{\mathbb{R}^{n-1}}\vert(\nabla_{x'}\eta)\vert^2~dx'.
\end{split}
\end{equation}
Hence the Proposition \ref{prop_1} follows.
\end{proof}

In section \ref{sec_3}, we have recovered $h$ at the boundary of $\Omega$, hence $h(x_0)$ is known to us. The cutoff function $\eta$ introduced in \eqref{defn_eta} is arbitrary, the right hand side term 
\begin{align*}
h(x_0) \int_{\mathbb{R}^{n-1}}\vert \nabla_{x'}\eta(x')\vert^2~dx' 
\end{align*}
of \eqref{final} is known. Hence, the normal derivative $\pl_{x_n} h(x_0)$ is determined. As a consequence, $\pl_{x_n}\gamma(x_0)$ can be reconstructed and hence the Theorem \ref{thm_2} follows.

   \section*{Acknowledgments}
			 MV is partially supported by Start-up Research Grant SRG/2021/001432  from the Science and Engineering Research Board, Government of India.


\begin{thebibliography}{}
				
				\bibitem{alessandrini2009local}
				G. Alessandrini and R. Gaburro;
				\newblock{The local Calderon problem and the determination at the boundary of the conductivity},
				\newblock{\em Communications in Partial Differential Equations}, 34 (2009), no. 8, 918--936.
				
				\bibitem{bt}
				T. Brander;
				\newblock{Calder{\'o}n problem for the $p$-Laplacian: First order derivative of conductivity on the boundary},
				\newblock{\em Proceedings of the American Mathematical Society}, 144 (2016), no. 1, 177--189.
				
				\bibitem{brown2001recovering}
				R. M. Brown;
				\newblock{Recovering the conductivity at the boundary from the Dirichlet to Neumann map: a pointwise result},
				\newblock{\em Journal of Inverse and Ill-posed Problems}, 9 (2001), no. 6, 567--574.
				
				\bibitem{carstea2019reconstruction}
				C. I. C{\^a}rstea, G. Nakamura and M. Vashisth;
				\newblock{Reconstruction for the coefficients of a quasilinear elliptic partial differential equation,}
				\newblock{\em Applied Mathematics Letters}, 98 (2019), 121--127.
				
				\bibitem{cm}
				C. I. C{\^a}rstea and M. Kar;
				\newblock{Recovery of coefficients for a weighted p-Laplacian perturbed by a linear second order term,}
				\newblock{\em Inverse Problems}, 37 (2020), no. 1, 015013.
				\bibitem{Ali_Lauri} A. Feizmohammadi and L. Oksanen;  An inverse problem for a semi-linear elliptic equation in Riemannian geometries, {\it Journal of Differential Equations}, 269(6):4683–4719, 2020.
    \bibitem{Catalin etal} C.I. C\^{a}rstea, A.  Feizmohammadi, Y. Kian, K. Krupchyk and G. Uhlmann;  The Calder\'on inverse problem for isotropic quasilinear conductivities, 
{\it Advances in Mathematics}, 391 (2021), Paper No. 107956, 31 pp.
				
				\bibitem{hannukainen2019inverse}
				A. Hannukainen, N. Hyv{\"o}nen and L. Mustonen;
				\newblock{An inverse boundary value problem for the p-Laplacian: a linearization approach},
				\newblock{\em Inverse Problems}, 35 (2019), no. 3, 034001.
    \bibitem{Isakov}  V. Isakov; Uniqueness of recovery of some quasilinear partial differential equations, \newblock{\em Communications in Partial Differential Equations}, 26 (2001),
1947-1973.
\bibitem{Isakov_Nachman}  V. Isakov and A. Nachman; Global uniqueness for a two-dimensional semilinear elliptic inverse problem, {\it Transaction of  American 
Mathematical Society}, 347 (1995), no. 9, 3375-3390.
    	\bibitem{Isakov_Sylvester}V. Isakov and J. Sylvester; Global uniqueness for a semi linear elliptic inverse problem,{ \it  Communications on Pure and  Applied  Mathematics}, 47
(1994), 1403-1410.


				\bibitem{KN} H. Kang and G. Nakamura; Identification of nonlinearity in a conductivity equation via the Dirichlet-to-Neumann map, {\it Inverse Problems}, 18 (2002), no. 4, 1079-1088.
				\bibitem{kang2002boundary}
				H. Kang and K. Yun;
				\newblock{Boundary determination of conductivities and Riemannian metrics via local Dirichlet-to-Neumann operator},
				\newblock{\em SIAM Journal on Mathematical Analysis}, 34 (2002), no. 3, 719--735.
    \bibitem{Kian} Y. Kian; Lipschitz and Hölder stable determination of nonlinear terms for elliptic equations, 
{\it Nonlinearity},  36 (2023), no. 2, 1302–1322.
				\bibitem{Kohn_Vogelius} R. Kohn and M. Vogelius;  Determining conductivity by boundary measurements, {\it Communications on Pure and Applied Mathematics},  37 (1984), no. 3, 289–298.
    \bibitem{KU-1} K. Krupchyk and G. Uhlmann;
Partial data inverse problems for semilinear elliptic equations with gradient nonlinearities, 
{\it Mathematics Research Letters}, 27 (2020), 1801-1824. 
\bibitem{KU-2} K. Krupchyk and G. Uhlmann;
A remark on partial data inverse problems for semilinear elliptic equations,
{\it Proceedings of the American Mathematical Society}, 148 (2020), 681-685.
				\bibitem{lee1989determining}
				J. M. Lee and G. Uhlmann;
				\newblock{Determining anisotropic real-analytic conductivities by boundary measurements},
				\newblock{\em Communications on Pure and Applied Mathematics}, 42 (1989), no. 8, 1097--1112.
				\bibitem{Lassas_Toni_Lin_Salo} M. Lassas, T. Liimatainen, Y.H. Lin and M. Salo; Inverse problems for elliptic equations with power type nonlinearities,
{\it Journal de Mathématiques Pures et Appliqu\'ees}, 145 (2021), 44-82.
				\bibitem{nac}
				A. I. Nachman;
				\newblock{Reconstructions from boundary measurements},
				\newblock{\em Annals of Mathematics}, 128 (1988), no. 3, 531--576.
				
				\bibitem{nt}
				G. Nakamura and K.Tanuma;
				\newblock{Direct determination of the derivatives of conductivity at the boundary from the localized Dirichlet to Neumann map},
				\newblock{\em Communications of the Korean Mathematical Society}, 16 (2001), No. 3, 415--425.
				
				\bibitem{ntt}
				G. Nakamura and K. Tanuma;
				\newblock{Local determination of conductivity at the boundary from Dirichlet to Neumann map},
				\newblock{\em Inverse Problems}, 17 (2001), 405–419.
				
				\bibitem{nakamura2005numerical}
				G. Nakamura, S. Siltanen, K. Tanuma and S. Wang;
				\newblock{Numerical recovery of conductivity at the boundary from the localized Dirichlet to Neumann map}, 
				\newblock{\em Computing}, 75 (2005), 197--213.
				
				%\bibitem{robertson1997boundary}
				%R. L. Robertson;
				%\newblock{Boundary identifiability of residual stress via the Dirichlet to Neumann map},
				%\newblock{\em Inverse Problems}, 13 (1997), no. 4, 1107--1119.
				
				\bibitem{salo2012inverse}
				M. Salo and X. Zhong;
				\newblock{An inverse problem for the p-Laplacian: boundary determination},
				\newblock{\em SIAM Journal on Mathematical Analysis}, 44 (2012), no. 4, 2474--2495.
				
		\bibitem{MIKO UHLMAN FELDMAN}
			J. Feldman, M. Salo and G. Uhlmann;
			\newblock{The calder\'on problem-An Introduction to Inverse Problems}.
				
				
				\bibitem{sylvester1988inverse}
				J. Sylvester and G. Uhlmann;
				\newblock{Inverse boundary value problems at the boundary—continuous dependence},
				\newblock{\em Communications on Pure and Applied Mathematics}, 41 (1988), no. 2, 197--219.
			
				
			\end{thebibliography}
		\end{document}